\documentclass[12pt,draftcls,onecolumn]{IEEEtran}

\usepackage{amsmath,amssymb,array,graphicx,palatino,verbatim}

\usepackage{latexsym,url}
\usepackage{amsmath,amssymb,amsthm}
\usepackage{algorithm,algorithmic}
\usepackage{color,array,graphicx,verbatim,xtab}
\usepackage{cite} % introduces sort and compress of multiple refs
\usepackage{setspace}

\usepackage
  [breaklinks,bookmarks,bookmarksnumbered,bookmarksopen,bookmarksopenlevel=2]
  {hyperref}

\newcommand{\reals}{{\mathbb{R}}}

\newcommand{\vect}[1]{\boldsymbol{#1}}
\newcommand{\ie}{{\it i.e.}}

\newtheorem{theorem}{Theorem}
\newtheorem{lemma}[theorem]{Lemma}

\newtheorem{remark}{Remark}

\usepackage{color}
\newcommand{\mycomment}[3]%
{%
\marginpar{%
  \hfil%
  \tiny{\textcolor{#2}{{\bf\textsc{#1}}}}%
  \hfil%
}%
\footnote{\textcolor{#2}{{\bf\textsc{#1}:}~~#3}}%}
}
\newcommand{\kakhbod}[1]%
{\mycomment{kakhbod}{blue}{#1}}
\newcommand{\jckoo}[1]%
{\mycomment{jckoo}{red}{#1}}

\begin{document}
\title{An Efficient T\^{a}tonnement Process for the Public Good Problem\\
\Large{A Decentralized Subgradient Approach}}
%\runtitle{[none yet]}
\author{\Large{Ali~Kakhbod$^1$, Joseph~Koo$^2$ and  Demosthenis~Teneketzis$^3$}\\ 
\large{${}^{1,3}$University of Michigan,  ${}^2$Stanford University} \\
\normalsize{Email: ${}^{1,3}${\tt{\{akakhbod,teneket\}@umich.edu}}, ${}^2${\tt{jckoo@stanford.edu}}}}

%\date{\today}

\maketitle

\begin{abstract}
We present a decentralized message exchange process (t\^{a}tonnement
process) for determining the level at which a certain public good will
be provided to a set of individuals who  finance the cost of attaining
that level.  The message exchange process we propose requires minimal
coordination overhead and converges to the optimal solution of the
corresponding centralized problem. 
\end{abstract}


\section{Introduction}
\label{sec:pg_intro}

The inspiration for this work comes from the following question:

%\jckoo{Is this copied verbatim from \cite{redondo:econ_theory_games}?
%We may need to write the problem statement in our own voice.  Or we
%need to specifically cite \cite{redondo:econ_theory_games} here.}
\begin{quotation}
\begin{it}
Suppose a community of individuals/agents have to determine the level
at which a certain public good will be provided to all of them.  For
example they have to decide on the quality of a public transportation
system or the resources devoted to running a public school.  The cost
of attaining any particular level of the public good has to be
financed by the individual/agent contributions.  Under the
assumption that the satisfaction (utility) of each
individual---which is a function of his contribution and the amount of
produced public good---is its private information, what is the
optimal value of the public good that maximizes the social welfare?
\end{it}
%\cite{redondo:econ_theory_games}.
\end{quotation}

This problem is called the public good
problem~\cite{mascolell:microeconomics, redondo:econ_theory_games},
and addresses the question of allocating individual resources toward a
public good in order to maximize the sum of the individuals'
utilities.  T\^{a}tonnement processes for solving the public good
problem have been proposed and analyzed in the economics literature,
\cite{laffont:planning_with_externalities,
dreze:tatonnement_public_goods,
malinvaud:procedures_collective_consumption,
tulkens:dynamic_proc_public_goods,
ledyard1971:non_tatonnement_process}.  They have also appeared in the
context of resource allocation problems in communication networks
\cite{sharma:externalities_decent_opt_power_alloc}. Decomposition
methods for convex optimization problems resembling the public good
problem have appeared in the engineering literature
\cite{bert_tsit:parallel_distributed_comp,
palomar:tutorial_decomp_num}.
%. A survey of these decomposition and subgradient methods can be
%found in \cite{palomar:tutorial_decomp_num,
%nedic:subgrad_methods_cvx_min, bertsekas:nonlinear_prog,
%boyd:convexopt}.

In this paper, we provide a decentralized method for solving the
public good problem. The method is different from those appearing in
\cite{laffont:planning_with_externalities,
dreze:tatonnement_public_goods,
malinvaud:procedures_collective_consumption,
tulkens:dynamic_proc_public_goods,
ledyard1971:non_tatonnement_process,
sharma:externalities_decent_opt_power_alloc}, satisfies the problem's
informational and resource constraints,  requires minimal coordination
overhead and converges to the optimal solution of the corresponding
centralized problem.  According to the method, the public good problem
is decomposed by first posing it as a convex optimization problem and
then deriving the respective dual decomposition.  We show that the
dual problem can be  separated into a single planner problem and
multiple agent subproblems.  The decomposition is amenable to
algorithms that have minimal messaging exchange overhead and satisfy
all the informational and resource constraints of the original
problem.  Moreover, the dual-decomposed problem can be solved using
simple subgradient methods, and the optimality of the solutions can be
guaranteed.
%\cite{bertsekas:nonlinear_prog}. Decompositions for convex
%optimization problems have previously been considered in the context
%of network utility problems---of which an excellent tutorial
%reference is \cite{nedic:subgrad_methods_cvx_min}.  Additional work
%and discussions on decompositions applied to network utility
%maximization problems can also be found in
%\cite{xiao:srra_dual_decomp, palomar:tutorial_decomp_num,
%lin:tutorial_cross_layer_opt}.

To the best of our knowledge, the approach to the solution of the public good problem  formulated/investigated  in this paper  provides a
new decomposition methodology that is simple to implement and is
physically meaningful.  Furthermore, the proposed solution methodology
does not require quasi-linearity or monotonicity properties of the
agents' utility functions.
%To the best of our knowledge our solution is the solution to the
%public good problem based on decompositions framework with the
%weakest constraints over the agents' utilities, which is simple to
%implement and is physically
%meaningful.
\\

The rest of the paper is organized as follows. In Section \ref{sec:pg_formulation} we formulate the public good problem we propose to investigate. In Section \ref{sec:pg_planner_mechanism} we present a t\^{a}tonnement process that satisfies  the informational and resource constraints of the problem and leads to the optimal solution of the centralized optimization problem corresponding to the problem of Section \ref{sec:pg_formulation}. We conclude in Section \ref{con}. The proofs of technical results appear in the Appendix.\\

\textit{Notation:} We use the following notation throughout the paper.
We denote by $\|\cdot\|$  the $\ell_2$-norm. We denote vectors by
bold-face symbols, as in $\vect{x}$. We denote by $\vect{x}^T$ the
transpose of $\vect{x}$.

\section{Problem Formulation}
\label{sec:pg_formulation}

Our problem formulation is as in \cite{redondo:econ_theory_games}.
Consider a community $\mathcal{A}$ of $m$ individuals/agents, $\mathcal{A} :=
\{1,2,\ldots,m\}$, who have to determine the level $x$ at which a
certain public good will be provided to all of them.
%(For example they have to decide on the quality of a public
%transportation system or the resources devoted to running a public
%school.)
The cost of attaining any particular level of the public good has to
be financed by individual/agent contributions $t_i$, $i=1,2,\ldots,m$,
of a private good; we may think of it as money, where $t_i$ stands for
the private contribution of individual $i$.  For simplicity, we assume
that the production of the public good displays constant returns, so
that a total contribution of $T \equiv \sum_{i=1}^m t_i$ may finance a
public good at level $x=T$.  Notice that, given constant returns, a
transformation rate between private and public good equal to one can
be simply obtained by a suitable choice of units, \ie, the units in
which money is expressed.

Let $w_i > 0$, $i=1,2,\ldots,m$, be the amount of private good
originally held by each individual/agent $i$.  The amount of $w_i$ is
agent $i$'s private information.  His preference over public good
levels and his own contribution are assumed to be represented by a
function of the form
\begin{equation}
U_i : \reals_+ \times (-\infty,w_i] \to \reals
\mbox{,}
\label{eq:public_good_def_Ui}
\end{equation}
that specifies the utility $U_i(x, t_i)$ enjoyed by agent $i$ when the
level of public good is $x$ and his individual contribution is $t_i$.
In principle, the individual contribution could be negative, which
would be interpreted as the receipt of (positive) transfers. The
function $U_i$ is agent $i$'s private information.  We assume that all
utility functions $U_i(\cdot,\cdot)$ are increasing in their first
argument, decreasing in their second argument, and jointly strictly
concave in $x$ and $t_i$.  We do not assume that  $U_i$, $i \in
\mathcal{A}$, is quasilinear or monotonic.

Now consider a planner who does not know the agents' utility functions
and the amounts $w_i$, $i \in \mathcal{A}$, of the agents' private
goods.  The planner's goal is to design a mechanism to select the
level $x$ of the public good and the individual contributions $t_i$,
$i \in \mathcal{A}$, so as to solve the weighted total utility
maximization problem
%whose preference may be represented by the different individuals.
%Denoting by $\vect{\alpha} = (\alpha_1, \alpha_2, \ldots, \alpha_m)$
%the vector of positive weights $\alpha_i$ that he attributes to each
%individual, $i = 1, 2, \ldots, m$, the planner's decision must be a
%solution to the following public good optimization problem:
\begin{equation}
\begin{array}{ll}
\mbox{maximize}
& \displaystyle \sum_{i=1}^m \alpha_i U_i(x,t_i) \\
\mbox{subject to}
& \displaystyle \sum_{j=1}^m t_j \geq x \\
& \displaystyle  0 \leq x \leq L \\
& \displaystyle t_i\leq w_i, \; \; \forall i \in \mathcal{A}
\end{array}
\label{eq:public_good_problem}
\mbox{,}
\end{equation} 
where for any $i \in \mathcal{A}$, $\alpha_i > 0$ and $\alpha_i$ is
known to the planner.  We assume that the feasible set is nonempty;
\ie, that there exists some $\hat{x}$ and $\hat{t_i}$, $i=1,\ldots,m$,
within the domain of the objective such that $\sum_{j=1}^m t_j \geq
x$, $0 \leq x \leq L$, and $t_i \leq w_i$ for all $i=1,\ldots,m$.

If the planner knew $U_i$ and $w_i$ for any $i$, he could obtain $x$
and $t_i$ for all $i\in \mathcal{A}$ by solving the mathematical
programming problem defined by \eqref{eq:public_good_problem}.  Since
$U_i$ and $w_i$ are agent $i$'s private information the planner has to
specify a mechanism, that is, a message space,  a message exchange process and an
allocation rule which determines $x$ and  $t_i$ for any $i \in
\mathcal{A}$, based on the outcome of the message exchange process.
The mechanism must be such that $x$ and $t_i$, $i \in \mathcal{A}$, is
a solution of the mathematical programming problem defined by
\eqref{eq:public_good_problem}.

%%%%%%%%%%%%%%%%%%%%%%%%%%%%%%%%%%
%%%%%%%%%%%%%%%%%%%%%%%%%%%%%%%%%%
\section{The Planner's Mechanism}
\label{sec:pg_planner_mechanism}

We specify a mechanism that satisfies the problem's informational and resource
constraints and results in an allocation $x$ and $t_i$ for all $i \in
\mathcal{A}$ that is a solution of problem
\eqref{eq:public_good_problem}.  The mechanism is defined by the
following t\^{a}tonnement process described by
Algorithm~\ref{alg:public_good_decomp} below.
%Here, $\epsilon > 0$ is some appropriately-chosen convergence
%threshold, the norm $\|\cdot\|$ in the convergence criterion is the
%$\ell_2$-norm, and we define $\vect{\lambda} : =(\lambda_1,
%\lambda_2, \ldots, \lambda_m)$ and $\vect{\mu} := (\mu_{ij} \; | \;
%i,j = 1,2,\ldots,m, ~\mbox{and} ~i > j)$
In this algorithm the parameters $\zeta_k$, $k = 0,1,2,\ldots$, are
chosen in a way similar to that in
\cite{nedic:incremental_subgrad_methods} so that
\[
\sum_{k=0}^{\infty} \zeta_k = \infty, \quad \quad \sum_{k=0}^{\infty}
\zeta_k^2 < \infty
\label{eq:zeta_const}
\]
(an example of such a sequence $\zeta_k$, $k=0,1,2,\ldots$, is
$\zeta_k=\frac{r}{k+1}$ where $r=\mbox{constant}$, $r>0$).
%need to satisfy certain conditions (given in the Appendix in order
%for the algorithm to converge.

\begin{algorithm}
\caption{The Planner's Mechanism}
\label{alg:public_good_decomp}
\begin{algorithmic}[1]
\STATE \label{it:pg_initialize}
Set $k := 0$.  The planner initializes $\vect{\lambda}(0) :=
(\lambda_1(0), \lambda_2(0), \ldots, \lambda_m(0)) = \vect{0}$, and
$\vect{\mu}(0) :=
(\mu_1(0), \mu_2(0), \ldots, \mu_m(0))$ is arbitrary and bounded.  The planner also
initializes $g_{\mathrm{min}} := \infty$.
\STATE \label{it:pg_pricing_policy}
The planner announces $\vect{\lambda}(k) = (\lambda_1(k),
\lambda_2(k), \ldots, \lambda_m(k))$, $\mu_{ij}(k)$ (for all $j < i$)
and $\mu_{ji}(k)$ (for all $j > i$) to each agent $i$, $i =
1,2,\ldots,m$.
\STATE \label{it:pg_agent_subproblem}
Agent $i$, $i=1,2,\ldots,m$, solves
\[
\begin{array}{ll}
\mbox{maximize}
& \displaystyle \alpha_i U_i(x_i,t_i) + t_i \sum_{j=1}^m \lambda_j(k)
- x_i \lambda_i(k) + x_i \left[ \sum_{j=1}^{i-1} \mu_{ij}(k) -
\sum_{j=i+1}^m \mu_{ji}(k) \right] \\
\mbox{subject to}
& \displaystyle 0 \leq x_i \leq L \\
& \displaystyle t_i \leq w_i
\end{array}
\]
for variables $x_i \in \reals$ and $t_i \in \reals$.  Let
$\bar{x}_i(k)$ and $\bar{t}_i(k)$ be a solution to agent $i$'s
problem, and let $g_i(\vect{\lambda}(k), \vect{\mu}(k))$ denote the
optimal value of agent $i$'s problem at this solution.  Agent $i$
announces $\bar{x}_i(k)$, $\bar{t}_i(k)$, and $g_i(\vect{\lambda}(k),
\vect{\mu}(k))$ to the planner.
\STATE \label{it:pg_update_gmin}
The planner computes
\[
g(\vect{\lambda}(k), \vect{\mu}(k)) = \sum_{i=1}^m
g_i(\vect{\lambda}(k), \vect{\mu}(k))
\mbox{.}
\]
If $g(\vect{\lambda}(k), \vect{\mu}(k)) \leq g_{\mathrm{min}}$, the
planner updates $g_{\mathrm{min}} := g(\vect{\lambda}(k),
\vect{\mu}(k))$ and $k_{\mathrm{min}} := k$.  Set
\begin{eqnarray*}
\vect{\lambda}_{\mathrm{min}}(k) & := & \vect{\lambda}(k_{\mathrm{min}})  \\
\vect{\mu}_{\mathrm{min}}(k) & := & \vect{\mu}(k_{\mathrm{min}})
\mbox{.}
\end{eqnarray*}
\STATE \label{it:pg_subgrad_update}
The planner updates $\lambda_i(k)$, $i = 1,2,\ldots,m$, according to
\[
\lambda_i(k+1) = \left[ \lambda_i(k) - \zeta_k \left[ \sum_{j=1}^m
\bar{t}_j(k) - \bar{x}_i(k) \right] \right]^+
\mbox{.}
\]
The planner also updates $\mu_{ij}(k)$, $i,j=1,2,\ldots,m$ where
$i>j$, according to
\[
\mu_{ij}(k+1) = \mu_{ij}(k) - \zeta_k [\bar{x}_i(k) - \bar{x}_j(k)]
\mbox{.}
\]
\STATE \label{it:pg_update}
Update $k := k + 1$ and go to Step~\ref{it:pg_pricing_policy}.
\STATE \label{it:pg_set_opt_solution}
At the stationary point, \ie, when convergence is achieved, the
planner sets $k^{\star} := k_{\mathrm{min}}$.  Then $x_i^{\star} :=
\bar{x}_i(k_{\mathrm{min}})$ and $t_i^{\star} :=
\bar{t}_i(k_{\mathrm{min}})$ for all $i \in \mathcal{A}$, where
$\bar{x}_i(k_{\mathrm{min}})$ and $\bar{t}_i(k_{\mathrm{min}})$ are
associated with $(\vect{\lambda}_{\mathrm{min}}(\infty),
\vect{\mu}_{\mathrm{min}}(\infty))$.  The planner also sets
$\vect{\lambda}^{\star} := \vect{\lambda}_{\mathrm{min}}(\infty)$ and
$\vect{\mu}^{\star} := \vect{\mu}_{\mathrm{min}}(\infty)$.  The
planner charges agent $i$ the amount
\[
\gamma_i(\vect{\lambda}^{\star}, \vect{\mu}^{\star}) = {x}_i^{\star}
\lambda_i^{\star} - {x}_i^{\star} \left[ \sum_{j=1}^{i-1}
\mu_{ij}^{\star} - \sum_{j=i+1}^m \mu_{ji}^{\star} \right] -
{t}_i^{\star} \sum_{j=1}^m \lambda_j^{\star}
\mbox{.}
\]
\end{algorithmic}
\end{algorithm}

The key features of Algorithm~\ref{alg:public_good_decomp} are
described by the following theorem.
\begin{theorem}
\label{thm:pg_alg_convergence}
Under the assumption that there exists at least one solution to the
dual of problem \eqref{eq:public_good_problem} that lies in a compact
subset $\mathcal{C}$ of $\mathbb{R}_+^m\times
\mathbb{R}^{\frac{m(m-1)}{2}}$ with diameter $\Lambda$,\footnotemark
~Algorithm~\ref{alg:public_good_decomp} has the following properties:
\footnotetext{The diameter of a subset of a metric space is the least
upper bound of the distances between pairs of points in the subset.}
\begin{itemize}
\item The limits satisfy
\begin{eqnarray}
\lim_{k \rightarrow \infty} \vect{\lambda}_{\mathrm{min}}(k) & = &
\vect{\lambda}^{\star} \\
\lim_{k \rightarrow \infty} \vect{\mu}_{\mathrm{min}}(k) & = &
\vect{\mu}^{\star} 
\end{eqnarray}
\item The level of public good $\vect{x}^{\star}$ and the taxes
$\vect{t}^{\star}:=(t_1^{\star},t_2^{\star},\ldots,t_m^{\star})$ that
correspond to $\vect{\lambda}^{\star}$ and $\vect{\mu}^{\star}$ and
result from Step~\ref{it:pg_agent_subproblem} of the algorithm are
solutions of problem~\eqref{eq:public_good_problem}.
\end{itemize}
\end{theorem}
\begin{proof}
See the appendix.
\end{proof}
\begin{remark}
\label{rem1}
The assumption under which the assertions of the theorem hold is not
restrictive.  In fact if, based on the problem, the compact set
$\mathcal{C}$ and its diameter $\Lambda$ are chosen appropriately by
the planner, a solution of the dual problem of
\eqref{eq:public_good_problem} will lie in $\mathcal{C}$.
\end{remark}

\begin{section}{Conclusion}
\label{con}
We presented a t\^{a}tonnement process to determine the level at which a certain public good must be provided to a set of individuals. The t\^{a}tonnement process satisfies the informational and resource constraints of the public good problem, requires minimal coordination overhead and converges to the optimal solution of the corresponding centralized problem. Furthermore, the proposed solution methodology dose not require quasi-linearity or monotonicity properties of the agents' utility functions.  
\end{section}

\bibliographystyle{IEEEtran} % use IEEEtran.bst style
\bibliography{IEEEabrv,references}

%%%%%%%%%%%%%%%%%%%%%%%%%%%%%%%%%%
%%%%%%%%%%%%%%%%%%%%%%%%%%%%%%%%%%

\appendix[\textbf{\textit{Proof of Theorem}}~\ref{thm:pg_alg_convergence}]
\label{app:proof_alg_convergence}

Note that for ease of exposition, we have set the subgradient update
coefficient $\zeta_k$ to be the same for both the $\vect{\lambda}$
updates and the $\vect{\mu}$ updates in
Step~\ref{it:pg_subgrad_update} of
Algorithm~\ref{alg:public_good_decomp}.  This is not strictly
necessary, as our proofs can be easily adapted---with some additional
bookkeeping---to the case when the two updates have different
coefficients.

%In this exposition, vectors are given in bold-face, as in $\vect{x}$,
%and transpose is denoted by $\vect{x}^T$.

%We first prove Theorem~\ref{thm:pg_alg_convergence} we need to
%establish the following notations and concepts.

We first redefine the utility functions in order to make some of the
problem's constraints implicit.  For all $i \in \mathcal{A}$, let
\begin{equation}
\tilde{U}_i(z,u) = \left\{ \begin{array}{ll}
U_i(z,u), & \mbox{if} ~0 \leq z \leq L ~\mbox{and} ~u \leq w_i \\
-\infty, & \mbox{otherwise}
\end{array} \right.
\label{eq:public_good_utility_redefined}
\mbox{.}
\end{equation}
Then problem \eqref{eq:public_good_problem} is equivalent to
\begin{equation}
\begin{array}{ll}
\mbox{maximize}
& \displaystyle \sum_{i=1}^m \alpha_i \tilde{U}_i(x,t_i) \\
\mbox{subject to}
& \displaystyle \sum_{j=1}^m t_j \geq x
\end{array}
\label{eq:public_good_implicit_constraints_problem}
\mbox{.}
\end{equation}
We expand problem \eqref{eq:public_good_problem} as follows:
\begin{equation}
\begin{array}{ll}
\mbox{maximize}
& \displaystyle \sum_{i=1}^m \alpha_i \tilde{U}_i(x_i,t_i) \\
\mbox{subject to}
& \displaystyle \sum_{j=1}^m t_j \geq x_i, \; \; \forall i \in
\mathcal{A} \\
& \displaystyle x_i = x_j, \; \; \forall (i, j) ~\mbox{such that} ~i
\ne j
\end{array}
\label{eq:expanded_public_good_problem}
\mbox{,}
\end{equation} 
where the optimization variables are now $x_i \in \reals$ and $t_i \in
\reals$, for all $i \in \mathcal{A}$.  Define $\vect{x} = (x_1, x_2,
\ldots, x_m)$, $\vect{t} = (t_1, t_2, \ldots, t_m)$,
$\vect{\lambda}=(\lambda_1,\lambda_2,\ldots,\lambda_m)$ and
$\vect{\mu}=(\mu_{ij} \; | \; i,j=1,2,\ldots,m, ~\mbox{and} ~i > j)$.

The Lagrangian of \eqref{eq:expanded_public_good_problem} is
\begin{eqnarray}
L(\vect{x}, \vect{t}, \vect{\lambda}, \vect{\mu}) 
& = & \sum_{i=1}^m \alpha_i \tilde{U}_i({x_i},t_i) + \sum_{i=1}^m
      \lambda_i \left[ \sum_{j=1}^m t_j - x_i \right] + \sum_{(i,j): i
      > j} \mu_{ij} \left[ x_i - x_j \right]  \\
& = & \sum_{i=1}^m \left( \alpha_i \tilde{U}_i(x_i,t_i) + t_i
      \sum_{j=1}^m \lambda_j - x_i \lambda_i + x_i \left[
      \sum_{j=1}^{i-1} \mu_{ij} - \sum_{j=i+1}^m \mu_{ji} \right]
      \right)
\mbox{,}
\end{eqnarray}
where $\lambda_i$, $i = 1,\ldots,m$, are the respective Lagrange
multipliers associated with the financing constraints $\sum_{j=1}^m
t_j \geq x_i$, and $\mu_{ij}$, $i,j = 1,\ldots,m$ and $i > j$,
represent the multipliers associated with the pairwise equalities $x_i
= x_j$.  Notice that considering all pairs $(i,j)$ where $i \ne j$, is
the same as considering all $i$ and $j$ such that $i > j$.  

Let $g(\vect{\lambda}, \vect{\mu}) = \sup_{\vect{x}, \vect{t}}
L(\vect{x}, \vect{t}, \vect{\lambda}, \vect{\mu})$. Then the dual to
problem~\eqref{eq:expanded_public_good_problem} is
\begin{equation}
\begin{array}{ll}
\mbox{minimize} & g(\vect{\lambda}, \vect{\mu}) \\
\mbox{subject to} & \lambda_i \geq 0, \; \; \forall i = 1,2,\ldots,m
\end{array}
\label{eq:public_good_dual_problem}
\mbox{,}
\end{equation}
where the variables are $\vect{\lambda} \in \mathbb{R}_+^m$ and
$\vect{\mu}\in \mathbb{R}^{\frac{m(m-1)}{2}}$.  By assumption, there
is at least one solution of \eqref{eq:public_good_dual_problem} that
lies in $\mathcal{C}$.
%\begin{remark}
%\label{gg}
%We confine ourselves to find  the solutions of
%\eqref{eq:public_good_dual_problem} in a compact subset of the real
%space, that the diameter\footnote{The diameter of a subset of a
%metric space is the least upper bound of the distances between pairs
%of points in the subset.} of the compact space is $\Lambda$ which is
%fixed and predefined.
%\end{remark}

We decompose $g(\vect{\lambda}, \vect{\mu})$ so that
$g(\vect{\lambda}, \vect{\mu}) = \sum_{i=1}^m g_i(\vect{\lambda},
\vect{\mu})$, where $g_i(\vect{\lambda},\vect{\mu}), i\in
\mathcal{A},$ is the optimal value of the problem
\begin{equation}
\mbox{maximize} \quad  \alpha_i \tilde{U}_i(x_i,t_i) + t_i
\sum_{j=1}^m \lambda_j - x_i \lambda_i + x_i \left[ \sum_{j=1}^{i-1}
\mu_{ij} - \sum_{j=i+1}^m \mu_{ji} \right]
\end{equation}
which is equivalent to the problem
\begin{equation}
\begin{array}{ll}
\mbox{maximize}
& \displaystyle \alpha_i U_i(x_i,t_i) + t_i \sum_{j=1}^m \lambda_j -
x_i \lambda_i + x_i \left[ \sum_{j=1}^{i-1} \mu_{ij} - \sum_{j=i+1}^m
\mu_{ji} \right] \\
\mbox{subject to}
& \displaystyle 0 \leq x_i \leq L \\
& \displaystyle t_i \leq w_i
\end{array}
\label{eq:public_good_agent_subproblem}
\mbox{,}
\end{equation}
where  $x_i \in \reals$ and $t_i \in \reals$ are the optimization
variables and $\vect{\lambda}$ and $\vect{\mu}$ are fixed.  We denote
by $\bar{x}_i$ and $\bar{t}_i$  the solution to problem
\eqref{eq:public_good_agent_subproblem}.
%---and is the same for both problems
%\eqref{eq:public_good_implicit_constraints_agent_subproblem} and
%\eqref{eq:public_good_agent_subproblem}.

We define 
\begin{equation}
s_i := \sum_{j=1}^m \bar{t}_j - \bar{x}_i, \; \; \forall i \in
\mathcal{A}
\label{eq:public_good_subgrad_lambda}
\mbox{,}
\end{equation}
and
\begin{equation}
r_{ij} := \bar{x}_i - \bar{x}_j, \; \; \forall (i,j) ~\mbox{such that}
~i>j
\label{eq:public_good_subgrad_mu}
\mbox{.}
\end{equation}
For a given set of solutions $\bar{\vect{x}} = (\bar{x}_1, \bar{x}_2,
\ldots, \bar{x}_m)$ and $\bar{\vect{t}} = (\bar{t}_1, \bar{t}_2,
\ldots, \bar{t}_m)$ to the agent subproblems
\eqref{eq:public_good_agent_subproblem}, the quantities $s_i$ and
$r_{ij}$ are the subgradients of $g(\vect{\lambda}, \vect{\mu})$ with
respect to $\lambda_i$ and $\mu_{ij}$, respectively, for $i,j =
1,2,\ldots,m$ and $i > j$.\footnotemark
\footnotetext{Thus, in Step~\ref{it:pg_subgrad_update} of
Algorithm~\ref{alg:public_good_decomp}, since the term $\bar{x}_i -
\bar{x}_j$ of \eqref{eq:public_good_subgrad_mu} gives the mismatch in
the level of public good as desired by the agents $i$ and $j$, the
update for $\mu_{ij}$ can be interpreted as the price of being at the
current level of mismatch; the subgradient update gives a mechanism
for incentivizing pairwise negotiations between agents $i$ and $j$ to
reach a consensus.}

To proceed with the proof we require the following intermediate
lemmas.  Here, we only state the lemmas and prove them later.
\begin{lemma}
\label{thm:upper_bound_subgrads}
There exists a $\Upsilon$ such that for any $k$,
\[
\|\vect{s}^{(k)}\|^2 + \|\vect{r}^{(k)}\|^2 \leq \Upsilon
\mbox{.}
\]
\end{lemma}

\begin{lemma}
\label{thm:properties_subgrad_update}
For any solution $(\vect{\lambda}',\vect{\mu}')$ of
\eqref{eq:public_good_dual_problem} in $\mathcal{C}$,
\begin{equation}
\lim_{k \rightarrow \infty} g(\vect{\lambda}_{\mathrm{min}}(k),
\vect{\mu}_{\mathrm{min}}(k)) = g(\vect{\lambda}', \vect{\mu}')
\mbox{.}
\end{equation}
(where $\vect{\lambda}_{\mathrm{min}}(k),
\vect{\mu}_{\mathrm{min}}(k)$ are defined in step 4 of Algorithm 1.)
\end{lemma}

%%%%%%%%%%%%%%%%%%%%%%%%%%%%%%%%%%%%%%%%%%%%%%%%%%%%

Now we complete the proof of Theorem~\ref{thm:pg_alg_convergence}
using the results of the above lemmas.\\

From Lemma~\ref{thm:properties_subgrad_update}, we know that
$g(\vect{\lambda}_{\mathrm{min}}(k), \vect{\mu}_{\mathrm{min}}(k))$
tends to $g(\vect{\lambda}', \vect{\mu}')$, where $(\vect{\lambda}',
\vect{\mu}')$ is an optimal solution of
\eqref{eq:public_good_dual_problem}.  At convergence, the solutions
${x}_i^{\star}$ and ${t}_i^{\star}$, which are the maximizers from
Step~\ref{it:pg_agent_subproblem} of
Algorithm~\ref{alg:public_good_decomp} when $\vect{\lambda} =
\vect{\lambda}^{\star}$ and $\vect{\mu} = \vect{\mu}^{\star}$, are the
same as the maximizers for the agent subproblem
\eqref{eq:public_good_agent_subproblem} (for all $i \in \mathcal{A}$).
Furthermore, at convergence we get $x_i^{\star}=x_j^{\star}=x^{\star}$
for all $i,j \in \mathcal{A}$, $j\neq i$ because of
Step~\ref{it:pg_subgrad_update}  of the algorithm and the fact that
$\zeta_k>0$ for all $k=1,2,\ldots$.  By strong duality,\footnotemark
~the dual value at the solution to \eqref{eq:public_good_dual_problem}
is the same as the optimal value for
\eqref{eq:expanded_public_good_problem}, which, in turn, is the same
as the optimal value for problem~\eqref{eq:public_good_problem}.
\footnotetext{In our problem strong duality holds because
problem~\eqref{eq:public_good_problem} is a convex problem and
Slater's condition~\cite{boyd:convexopt} is satisfied.} Because the
objective function for \eqref{eq:public_good_agent_subproblem} is, by assumption,
strictly concave, the optimal solution for each agent's subproblem
\eqref{eq:public_good_agent_subproblem} is unique, and so
${x}^{\star}$ is the solution to the
problem~\eqref{eq:public_good_problem}.

We now proceed to prove Lemmas \ref{thm:upper_bound_subgrads} and \ref{thm:properties_subgrad_update}.

\textbf{Proof of Lemma \ref{thm:upper_bound_subgrads}}

\begin{proof}
For any $i \in \mathcal{A}$, $0 \leq \bar{x}_i \leq L$.  Also since
$\bar{t}_i \leq w_i$, we have
\begin{equation}
0 \leq \bar{x}_i \leq \sum_{j \in \mathcal{A}} \bar{t}_j = \bar{t}_i +
\sum_{\substack{j \in \mathcal{A} \\ j \neq i}} \bar{t}_j \leq
\bar{t}_i + \sum_{\substack{j \in \mathcal{A} \\ j \neq i}} w_j
\mbox{.}
\label{eq:derive_bounds_ti}
\end{equation}
Rearranging \eqref{eq:derive_bounds_ti} we obtain
\begin{equation}
-\sum_{\substack{j \in \mathcal{A} \\ j \neq i}} w_j \leq \bar{t}_i
\leq w_i, \quad \forall i \in \mathcal{A}
\mbox{.}
\label{eq:bounds_ti}
\end{equation}
Now, by using \eqref{eq:bounds_ti} and the definitions of $s_i$ and
$r_{ij}$ (equations \eqref{eq:public_good_subgrad_lambda} and
\eqref{eq:public_good_subgrad_mu}, respectively), we can show that, for any
$i=1,2,\ldots,m$,
\begin{eqnarray}
|s_i| & \leq & L+ m \max_{1 \leq j \leq m} w_j \\
|r_{ij}| & \leq & 2 L
\mbox{.}
\end{eqnarray}
Therefore,
\begin{eqnarray}
||\vect{s}^{(k)}||^2 + ||\vect{r}^{(k)}||^2
& \leq & \left(m (L+ m \max_{1 \leq j \leq m} w_j)^2\right) + m(2L)^2 \nonumber \\
& \leq & (4m+m(m+1)^2)\left[\max\{L,w_1,w_2,\ldots,w_m\}\right]^2:=\Upsilon
\mbox{.}
\end{eqnarray}
\end{proof}

\textbf{Proof of Lemma \ref{thm:properties_subgrad_update}}

\begin{proof}
Consider $(\vect{\lambda}', \vect{\mu}')$ as a solution of
\eqref{eq:public_good_dual_problem}. 
% Let $\Lambda$ be an upper bound
%on $\| \vect{\lambda}(0) - \vect{\lambda}' \|^2 + \| \vect{\mu}(0) -
%\vect{\mu}' \|^2$.  
Then
\begin{eqnarray}
\lefteqn{ \|(\vect{\lambda}(k+1), \vect{\mu}(k+1)) - (\vect{\lambda}',
\vect{\mu}')\|^2 }  \nonumber \\
& = & \|\vect{\lambda}(k+1) - \vect{\lambda}'\|^2 + \|\vect{\mu}(k+1)
- \vect{\mu}'\|^2  \nonumber \\
& = & \|\left[ \vect{\lambda}(k) - \zeta_k \vect{s}^{(k)} \right]^+ -
\vect{\lambda}'\|^2 + \|\vect{\mu}(k) - \zeta_k \vect{r}^{(k)} -
\vect{\mu}'\|^2  \nonumber \\
& \leq & \|\vect{\lambda}(k) - \zeta_k \vect{s}^{(k)} -
\vect{\lambda}'\|^2 + \|\vect{\mu}(k) - \zeta_k \vect{r}^{(k)} -
\vect{\mu}'\|^2  \nonumber \\
& = & \|\vect{\lambda}(k) - \vect{\lambda}'\|^2 + \|\vect{\mu}(k) -
\vect{\mu}'\|^2 + \zeta_k^2 \left( \|\vect{s}^{(k)}\|^2 +
\|\vect{r}^{(k)}\|^2 \right)  \nonumber \\
&& {-} \: 2 \zeta_k \left( [\vect{s}^{(k)}]^T (\vect{\lambda}(k) -
\vect{\lambda}') + [\vect{r}^{(k)}]^T (\vect{\mu}(k) - \vect{\mu}')
\right)  \nonumber \\
& \leq & \|\vect{\lambda}(k) - \vect{\lambda}'\|^2 + \|\vect{\mu}(k)
- \vect{\mu}'\|^2 + \zeta_k^2 \left( \|\vect{s}^{(k)}\|^2 +
\|\vect{r}^{(k)}\|^2 \right)  \nonumber \\
&& {-} \: 2 \zeta_k \left( g(\vect{\lambda}(k), \vect{\mu}(k)) -
g(\vect{\lambda}', \vect{\mu}') \right)
\label{eq:dual_vars_contract_expand_s_g} \\
& \leq & \|\vect{\lambda}(k) - \vect{\lambda}'\|^2 + \|\vect{\mu}(k) -
\vect{\mu}'\|^2 + \zeta_k^2 \Upsilon - 2 \zeta_k \left(
g(\vect{\lambda}(k), \vect{\mu}(k)) - g(\vect{\lambda}', \vect{\mu}')
\right)
\label{eq:dual_vars_upper_bound_subgrads}
%
%& \leq & \Lambda + \zeta_k^2 \Upsilon - 2 \zeta_k \left(
%g(\vect{\lambda}(k), \vect{\mu}(k)) - g(\vect{\lambda}', \vect{\mu}')
%\right)  \label{eq:dual_vars_contract_expand_s_g(3)}
\mbox{,}
\end{eqnarray}
where \eqref{eq:dual_vars_contract_expand_s_g} is due to the
definition of subgradient\footnotemark ~of $g$ at
$(\vect{\lambda}(k),\vect{\mu}(k))$, and
\eqref{eq:dual_vars_upper_bound_subgrads} results from
Lemma~\ref{thm:upper_bound_subgrads}.
% and, \eqref{eq:dual_vars_contract_expand_s_g(3)} implied by
% boundedness of the $(\vect{\lambda}', \vect{\mu}')$.
\footnotetext{A subgradient of $g$ at a point $\vect{x}$ is a vector
$\vect{s}$ that satisfies the inequality $g(\vect{y}) - g(\vect{x})
\geq \vect{s}^T (\vect{y} - \vect{x})$ for all $\vect{y}$.}
Repeatedly using \eqref{eq:dual_vars_upper_bound_subgrads} along with the assumption $\|\vect{\lambda}(0) -
\vect{\lambda}'\|^2 + \|\vect{\mu}(0) - \vect{\mu}'\|^2 \leq \Lambda$ (Remark \ref{rem1})  we obtain

\begin{eqnarray}
0
& \leq & \|(\vect{\lambda}(k+1), \vect{\mu}(k+1)) - (\vect{\lambda}',
\vect{\mu}')\|^2 \nonumber \\
& \leq & \Lambda + \sum_{i=0}^k \zeta_i^2 \Upsilon  - 2 \sum_{i=0}^k
\zeta_i \left( g(\vect{\lambda}(i), \vect{\mu}(i)) -
g(\vect{\lambda}', \vect{\mu}') \right)
\label{eq:dual_function_converge_recursion}
\mbox{.}
\end{eqnarray}

Since
\begin{equation}
\sum_{i=0}^k \zeta_i \left( g(\vect{\lambda}(i), \vect{\mu}(i)) -
g(\vect{\lambda}', \vect{\mu}') \right) \geq \left( \min_{0 \leq i
\leq k} g(\vect{\lambda}(i), \vect{\mu}(i)) - g(\vect{\lambda}',
\vect{\mu}') \right) \sum_{i=0}^k \zeta_i
\mbox{,}
\label{eq:dual_vars_contract_condition}
\end{equation}
\eqref{eq:dual_vars_contract_condition} along with
 \eqref{eq:dual_function_converge_recursion} imply that
\begin{equation}
0 \leq \Lambda + \sum_{i=0}^k \zeta_k^2 \Upsilon  - 2 \left( \min_{0
\leq i \leq k} g(\vect{\lambda}(i), \vect{\mu}(i)) -
g(\vect{\lambda}', \vect{\mu}') \right) \sum_{i=0}^k \zeta_i
\mbox{.}
\end{equation}
Therefore,
\begin{equation}
\min_{0 \leq i \leq k} g(\vect{\lambda}(i), \vect{\mu}(i)) -
g(\vect{\lambda}', \vect{\mu}') \leq \frac{\Lambda + \sum_{i=0}^k
\zeta_i^2 \Upsilon}{2 \sum_{i=0}^k \zeta_i}
\end{equation}
and consequently,
\begin{eqnarray}
\label{temp23}
\lim_{k \rightarrow \infty} \left( \min_{0 \leq i \leq k}
g(\vect{\lambda}(i), \vect{\mu}(i)) - g(\vect{\lambda}', \vect{\mu}')
\right)
 \leq  \lim_{k \rightarrow \infty} \left( \frac{\Lambda +
\sum_{i=0}^k \zeta_i^2 \Upsilon}{2 \sum_{i=0}^k \zeta_i} \right)
\mbox{.}
\end{eqnarray}
Since $\zeta_i, i=1,2,\cdots$ are such that $\sum_{i=0}^{\infty} \zeta_i = \infty$ and $\sum_{i=0}^{\infty} \zeta_i <
\infty$, 
\[
\lim_{k \rightarrow \infty} \left( \frac{\Lambda + \sum_{i=0}^k
\zeta_i^2 \Upsilon}{2 \sum_{i=0}^k \zeta_i} \right) = 0
\mbox{,}
\]
and,  therefore, by \eqref{temp23},
\begin{equation}
\lim_{k \rightarrow \infty} \left(\min_{0 \leq i \leq k}
g(\vect{\lambda}(i), \vect{\mu}(i)) - g(\vect{\lambda}',
\vect{\mu}')\right) = 0
\mbox{,}
\end{equation}
\ie,
\begin{equation}
\lim_{k \rightarrow \infty} g(\vect{\lambda}_{\mathrm{min}}(k),
\vect{\mu}_{\mathrm{min}}(k)) = g(\vect{\lambda}', \vect{\mu}')
\mbox{.}
\end{equation}
\end{proof}

\end{document}